\documentclass[a4paper,10pt,twoside]{article}

\usepackage{tensor}
\usepackage[utf8]{inputenc}
\usepackage[UKenglish]{babel}
\usepackage{cite}
\usepackage{amssymb}
\usepackage{amsmath}
\usepackage{mathrsfs}
\usepackage{amsthm}
\usepackage[T1]{fontenc}
\usepackage{accents}
\usepackage[cm]{fullpage}
\usepackage{graphicx}
\usepackage{tensor}
\usepackage{relsize}
\usepackage{appendix}
\usepackage[affil-it]{authblk}
\usepackage{tensor}
\usepackage{verbatim}
\usepackage{empheq}

\allowdisplaybreaks

\def\mequal{\mathrel{\mathpalette\@mvereq{\hbox{\sevenrm m}}}}

\usepackage{mathtools}
\mathtoolsset{showonlyrefs,showmanualtags}

\RequirePackage{lastpage}
\RequirePackage{latexsym}
\makeatletter
\ifx\@NODS\undefined\RequirePackage{dsfont}\fi
\ifx\@NOAMS\undefined\RequirePackage{amsmath,amsfonts,amssymb,amsthm}\fi
\makeatother

\RequirePackage{geometry}
\RequirePackage{bera}
\RequirePackage[pdftex,pagebackref=false]{hyperref}
\hypersetup{pdfborder=0 0 0}
\hypersetup{pdfstartview={FitH}}

\title{{\Large \bfseries{Exponential integrability and exit times of diffusions on sub-Riemannian and metric measure spaces}}}

\author{Anton Thalmaier and James Thompson\footnote{University of Luxembourg, Email: \texttt{james.thompson@uni.lu}}}

\date{\today}

\makeatletter
\def\@MRExtract#1 #2!{#1}
\newcommand{\MR}[1]{
  \xdef\@MRSTRIP{\@MRExtract#1 !}%
  \href{http://www.ams.org/mathscinet-getitem?mr=\@MRSTRIP}{MR-\@MRSTRIP}}
\makeatother

\makeatletter
\renewenvironment{thebibliography}[1]{%
  \section*{\refname
    \@mkboth{\MakeUppercase\refname}{\MakeUppercase\refname}}%
  \phantomsection%
  \addcontentsline{toc}{section}{\refname}%
  \list{\@biblabel{\@arabic\c@enumiv}}%
  {\settowidth\labelwidth{\@biblabel{#1}}%
    \small%
    \setlength{\labelsep}{0.4em}%
    \setlength{\leftmargin}{\labelwidth}%
    \addtolength{\leftmargin}{\labelsep}%
    \setlength{\itemsep}{-.25em}%
    \@openbib@code
    \usecounter{enumiv}%
    \let\p@enumiv\@empty
    \renewcommand\theenumiv{\@arabic\c@enumiv}}%
  \sloppy\clubpenalty4000\@clubpenalty\clubpenalty\widowpenalty4000%
  \sfcode`\.\@m}{%
  \def\@noitemerr{%
    \@latex@warning{Empty `thebibliography' environment}}%
  \endlist}
\makeatother

\makeatletter
\ifx\@NODS\undefined%

\let\mathbb=\mathds
\else%
\fi
\makeatother

\DeclareMathOperator{\Cut}{Cut}

\DeclareMathOperator{\m}{\mathfrak{m}}
\DeclareMathOperator{\supp}{supp}

\def\<{\langle}
\def\>{\rangle}

\def\Ricbf{\mathop{\mathbf{Ric}}}

\newcommand{\dimm}{m}

\newcommand{\sv}{k}
\newcommand{\mom}{p}
\newcommand{\nv}{h}

\newcommand{\E}{\mathbb{E}}
\newcommand{\R}{\mathbb{R}}

\DeclarePairedDelimiter\floor{\lfloor}{\rfloor}

\newtheorem{theorem}{Theorem}[section]
\newtheorem{lemma}[theorem]{Lemma}
\newtheorem{proposition}[theorem]{Proposition}
\newtheorem{corollary}[theorem]{Corollary}
\newtheorem{definition}[theorem]{Definition}

\begin{document}

\maketitle

\begin{abstract}%
   \noindent%
   {In this article we derive moment estimates, exponential integrability, concentration inequalities and exit times estimates for canonical diffusions in two settings each beyond the scope of Riemannian geometry. Firstly, we consider sub-Riemannian limits of Riemannian foliations. Secondly, we consider the non-smooth setting of $\mathrm{RCD}^*(K,N)$ spaces. In each case the necessary ingredients are an It\^{o} formula and a comparison theorem for the Laplacian, for which we refer to the recent literature. As an application, we derive pointwise Carmona-type estimates on eigenfunctions of Schr\"{o}dinger operators.}\\[1em]%
   {\footnotesize%
     \par%
     \par
   }
 \end{abstract}

\section*{Introduction}

Suppose a diffusion operator $L$ satisfies an inequality of the form $L \phi \leq \nu + \lambda \phi$ for some constants $\nu,\lambda$ and a suitable function $\phi$. Then, given either a suitable Kolmogorov equation or, better yet, an It\^{o} formula for the corresponding diffusion $X$, one should expect to be able to calculate various estimates on the moments of the random variable $\phi(X_t)$. In the recent article \cite{Thompson2016}, the second author considered the case of a complete Riemannian manifold with $L=\frac{1}{2}\Delta + Z$ where $Z$ is a smooth vector field and $\Delta$ the Laplace-Beltrami operator. The function $\phi$ was taken to be the square of the distance to either a fixed point or, more generally, a submanifold. Geometric conditions were given under which suitable constants $\nu,\lambda$ could be determined explicitly. In addition to moment estimates and an exponential integrability condition, a concentration inequality and exit time estimate for tubular neighbourhoods were also derived. In this paper, we consider two further situations where the ingredients required for such calculations are also available.

Firstly, in Section \ref{sec:subriem}, we look at the sub-Riemannian limit as $\epsilon \downarrow 0$ of a sequence of Riemannian metrics $g_\epsilon = g_{\mathcal{H}} \oplus \frac{1}{\epsilon} g_{\mathcal{V}}$ and assume a condition of the form
\begin{equation}\label{eq:quadbound}
\tfrac{1}{2}\Delta_{\mathcal{H}} r_\epsilon^2 \leq {\nu_\epsilon} + {\lambda_\epsilon} r_\epsilon^2
\end{equation}
for the horizontal Laplacian $\Delta_{\mathcal{H}}$. Here $r_\epsilon$ is the distance to some fixed point $x_0$ with respect to the metric $g_\epsilon$. Example \ref{ex:reeb} illustrates that such constants can be found if the Riemannian foliation is given by the Reeb foliation of a Sasakian structure with curvature bounded below. This observation is based on the Laplacian comparison theorem recently proved by Baudoin, Grong, Kuwada and Thalmaier in \cite{BGKT2017}. The It\^{o} formula is taken from their next article \cite{BGKT2018}, currently in preparation. The main results in Section \ref{sec:subriem} include the exponential estimate Theorem \ref{th:expenbm}, which for the subelliptic diffusion process $X$ generated by $\Delta_{\mathcal{H}}$ implies
\begin{equation}
\mathbb{E}^x \left[e^{\frac{\theta}{2} r_0^2 (X_t)}\right] \leq \left(1-\theta t\Lambda(t)\right)^{-\frac{{\nu}}{2}}\exp\left(\frac{\theta r_0^2(x)e^{{\lambda} t}}{2(1-\theta t \Lambda(t))}\right)
\end{equation}
for all $t,\theta \geq 0$ such that $\theta t\Lambda(t)<1$ where ${\Lambda}(t):=(e^{{\lambda} t}-1)/{\lambda t}$. Here $r_0, \nu$ and $\lambda$ are given as limits of $r_\epsilon, \nu_\epsilon$ and $\lambda_\epsilon$ as $\epsilon \downarrow 0$. There is also the concentration inequality Theorem \ref{th:concenthm}, which states for the sub-Riemannian ball $B^0_r(x_0)$ that
\begin{equation}
\lim_{r\rightarrow \infty} \frac{1}{r^2}\log \mathbb{P}^x \lbrace X_t \notin B^0_r(x_0) \rbrace \leq -\frac{1}{2t\Lambda(t)}
\end{equation}
and the exit time estimate of Theorem \ref{thm:exittime}, which is to the best of the authors' knowledge the first such exit time estimate proved in the sub-Riemannian setting.

Secondly, in Section \ref{sec:mms}, we look at the general setting of a metric measure space $(X,d,\mathfrak{m})$ satisfying the $\mathrm{RCD}^*(K,N)$ condition, with associated canonical diffusion $X$. The It\^{o} formula for the distance function, which in this case has a Laplacian comparison built into it, has been proved by Kuwada and Kuwae and can be found in \cite{KK2018}. This yields an exponential estimate, concentration inequality and, finally, exit time estimate
\begin{equation}
\mathbb{P}^x\bigg\lbrace \sup_{s\in \left[0,t\right]} r(X_s) \geq r \bigg\rbrace\leq (1-\delta)^{-\frac{{N+\lambda}}{2}}\exp\left(\frac{r^2(x)\delta e^{\lambda t}}{2(1-\delta)t\Lambda(t)}-\frac{\delta r^2}{2t\Lambda(t)}\right)
\end{equation}
for all $t>0$ and $\delta \in (0,1)$ where $\lambda := \frac{1}{2}\sqrt{(N-1)K^-}$. Here again $r(y) :=d(y,x_0)$ for some fixed point $x_0$. For the case $K=0$, we give an upper bound for the law of the iterated logarithm. Following this we present, as an application, Carmona-type upper estimates on Schr\"{o}dinger eigenfunctions. To do so we start, in Subsection \ref{ss:Feynman}, by describing conditions under which for each $\rho>0$ and $p>1$ there exists a positive constant $C_3(p,\rho)$ such that if $V \in L^p(X)$ is non-negative then
\begin{equation}
\E^x \left[\exp\left(\int_0^t V(X_s)ds\right)\right] \leq \exp\left( \rho\left(  t\, C_3(p,\rho) \left(1\vee\|V\|_{p}^{\frac{1}{\gamma}}\right) +1 \right)\right)
\end{equation}
for all $t\geq 0$ and all $x \in \supp \m$. The constant can be determined precisely, by Theorem \ref{thm:lpestv}. Combined with the exit time estimate described above, we then, in Subsection \ref{ss:upperbnds}, deduce the upper estimates for the eigenfunctions.

It would be desirable to find a general framework that ecompasses all three settings mentioned above (namely, the distance to a submanifold, considered in \cite{Thompson2016}, the sub-Riemannian distance, considered in Section \ref{sec:subriem} and the distance in an $\mathrm{RCD}^*(K,N)$ space, considered in Section \ref{sec:mms}). However, for the time being, no all-encompasing It\^{o} formula nor comparison theorem is to be found in the literature. Moreover, while the $\mathrm{RCD}^*(K,N)$ condition generalizes the concept of a Ricci lower bound $K$, the assumptions in \cite{Thompson2016} allow for unbounded curvature and in principle the same is true for the sub-Riemannian setting considered in Section \ref{sec:subriem}. Indeed, condition \eqref{eq:quadbound} is still satisfied if the curvature operators are bounded below merely by a negative quadratic in the distance function (as in \cite[Theorem 3.11]{BGKT2018}). So, instead of attempting to formulate a general approach to moment estimates for Markov processes, we focus on the two examples outlined above: sub-Riemannian limits of Riemannian foliations and $\mathrm{RCD}^*(K, N)$ spaces.

The authors wish to thank Erlend Grong for helpful discussions concerning Section \ref{sec:subriem} and Batu G\"{u}neysu for suggesting the application to Schr\"{o}dinger eigenfunctions.

\section{Riemannian foliations}\label{sec:subriem}

Suppose $(M,g)$ is a complete and connected Riemannian manifold of dimension $n+m$ equipped with a Riemannian foliation $\mathcal{F}$ having $m$-dimensional totally geodesic leaves. Let $\mathcal{V}$ be the integrable sub-bundle tangent to the leaves of $\mathcal{F}$ and denote by $\mathcal{H}$ its orthogonal complement with respect to $g$. Consider the canonical variation $g_\epsilon$ defined by
\begin{equation}
g_\epsilon := g_{\mathcal{H}} \oplus \frac{1}{\epsilon} g_{\mathcal{V}}
\end{equation}
where $g_{\mathcal{H}} := g \vert \mathcal{H}$ and $g_{\mathcal{V}} := g \vert \mathcal{V}$ for $\epsilon >0$. The limit $\epsilon \downarrow 0$ is called the sub-Riemannian limit. For each $\epsilon>0$, the Riemannian distance associated with $g_\epsilon$ will be denoted $d_\epsilon$. As $\epsilon \downarrow 0$, these distances form an increasing sequence converging pointwise to the sub-Riemannian distance $d_0$. Now let $x_0 \in M$ be fixed and for $\epsilon \geq 0$ denote
\begin{equation}
r_\epsilon(x) := d_\epsilon(x_0,x).
\end{equation}
The cut-locus $\Cut_\epsilon (x_0)$ of $x_0$ for the distance $d_\epsilon$ is defined as the complement of the set of points $y$ in $M$ for which there exists a unique length minimizing geodesic connecting $x_0$ with $y$ and such that $x_0$ and $y$ are not conjugate. The global cut locus $\Cut_\epsilon(M)$ is defined by
\begin{equation}
\Cut_\epsilon(M) := \lbrace (x,y) \in M \times M: y \in \Cut_\epsilon (x) \rbrace.
\end{equation}
It is well-known that the set $M \setminus \Cut_\epsilon(x_0)$ is open and dense in $M$, and that the function $d_\epsilon^2$ is smooth on $(M\times M)\setminus \Cut_\epsilon (M)$.

\subsection{Comparison theorems}

If $\nabla$ denotes the Riemannian gradient determined by $g$ then the projection of $\nabla$ to $\mathcal{H}$ will be denoted by $\nabla_{\mathcal{H}}$ and called the horizontal gradient. The horizontal Laplacian $\Delta_{\mathcal{H}}$ is then the generator of the symmetric closable bilinear form
\begin{equation}
\mathcal{E}_{\mathcal{H}}(f,g) = -\int_M g_{\mathcal{H}}(\nabla_{\mathcal{H}} f,\nabla_{\mathcal{H}} g) d\mu
\end{equation}
where $\mu$ denotes the Riemannian measure determined by $g$. We will suppose that there exist constants ${\nu_\epsilon} \geq 1$ and ${\lambda_\epsilon} \in \mathbb{R}$ such that inequality \eqref{eq:quadbound} holds on $M\setminus \Cut_\epsilon (x_0)$. Precise geometric conditions under which inequality \eqref{eq:quadbound} holds can be derived using, for example, comparison theorems for the horizontal Laplacian $\Delta_{\mathcal{H}}$ such as those presented in \cite{BGKT2017} for foliations of a Sasakian type, as explained in the following example.

\paragraph*{Example 1.1}\label{ex:reeb}
Suppose that the Riemannian foliation is the Reeb foliation of a Sasakian structure. Denote the Reeb vector field by $S$, the complex structure by $\mathbf{J}$ and denote by ${\Ricbf}_\mathcal{H}$ the horizontal Ricci curvature of the Bott connection. Furthermore, for $X \in \Gamma^\infty(\mathcal{H})$ with $\| X\| = 1$, set
\begin{equation}
\mathbf{K}_{\mathcal{H},J}(X,X) := \langle R(X,\mathbf{J}X)\mathbf{J}X,X\rangle_{\mathcal{H}},
\end{equation}
a quantity that is sometimes called the pseudo-Hermitian sectional curvature of the Sasakian manifold, and define
\begin{equation}
{\Ricbf}_{\mathcal{H},J^\perp}(X,X) := {\Ricbf}_\mathcal{H}(X,X) - \mathbf{K}_{\mathcal{H},J}(X,X).
\end{equation}
Suppose $k_1,k_2 \in \mathbb{R}$ are constants such that
\begin{equation}
\mathbf{K}_{\mathcal{H},J}(X,X) \geq k_1, \quad {\Ricbf}_{\mathcal{H},J^\perp}(X,X) \geq (n-2)k_2
\end{equation}
for all $X \in \Gamma^\infty(\mathcal{H})$ with $\| X\| = 1$. Then, in terms of the functions
\begin{equation}
\phi_{\mu}(r) :=
\begin{cases}
\frac{\sinh \sqrt{\mu} r}{\sqrt{\mu}} & \text{if } \mu >0 \\
r & \text{if } \mu =0 \\
\frac{\sin \sqrt{|\mu|} r}{\sqrt{|\mu|}} & \text{if } \mu <0 \\
\end{cases},
\quad
\Psi_{\mu}(r) :=
\begin{cases}
\frac{1}{\mu^{3/2}}\left( \sqrt{\mu}-\frac{1}{r}\tanh \sqrt{\mu}r\right) & \text{if } \mu >0 \\
\frac{1}{3} r^2 & \text{if } \mu =0 \\
\frac{1}{|\mu|^{3/2}}\left( \frac{1}{r}\tan \sqrt{|\mu|}r - \sqrt{|\mu|}\right) & \text{if } \mu <0 \\
\end{cases},
\end{equation}
and ${\nv}_\epsilon := \| \nabla_{\mathcal{H}} r_\epsilon \|^2$, \cite[Theorem~3.7]{BGKT2017} states:
\begin{equation}
\Delta_{\mathcal{H}} r_\epsilon \leq \frac{1}{r_\epsilon}\min \bigg \lbrace 1,\frac{1}{\nv_\epsilon}-1\bigg \rbrace + (n-2)\frac{\phi'_{-\nv_\epsilon k_2}(r_\epsilon)}{\phi_{-\nv_\epsilon k_2}(r_\epsilon)} + \frac{\phi'_{-\nv_\epsilon k_1}(r_\epsilon)}{\phi_{-\nv_\epsilon k_1}(r_\epsilon)} \frac{\nv_\epsilon \Psi_{-\nv_\epsilon k_1}(r_\epsilon)+\epsilon}{\nv_\epsilon \Psi_{-\nv_\epsilon k_1}(r_\epsilon /2)+\epsilon}
\end{equation}
on $M\setminus {\Cut}_\epsilon(x_0)$ and where $\nv_\epsilon >0$. In \cite[Theorem~3.1]{BGKT2017} it is shown how taking the limit as $\epsilon \downarrow 0$ produces a comparison theorem for the sub-Riemannian distance $r_0$. We, however, will continue to work with the $r_\epsilon$ distance since this is the one to which the It\^{o} formula of the next sub-section is applied. We must now therefore deduce from the above comparison an inequality of the type \eqref{eq:quadbound}. Firstly, for all non-negatively curved Sasakian foliations, in the sense that $\mathbf{K}_{\mathcal{H},J} \geq 0$ and ${\Ricbf}_{\mathcal{H},J^\perp} \geq 0$, it follows that
\begin{equation}\label{eq:nonneqsasak}
\tfrac{1}{2} \Delta_{\mathcal{H}} r_\epsilon^2 \leq \min \bigg\lbrace 1, \frac{1}{{\nv}_\epsilon} -1\bigg\rbrace + {\nv}_\epsilon + n + 2
\end{equation}
in which case one can choose $\lambda_\epsilon = 0$ and set $\nu_\epsilon$ equal to the supremum of the right-hand side. Note in fact that the right-hand side is bounded above by $n + \tfrac{7}{2}$, since $\nv_\epsilon \leq 1$. Alternatively, if
\begin{align}
\mathbf{K}_{\mathcal{H},J} \geq k,\quad {\Ricbf}_{\mathcal{H},J^\perp} \geq (n-2)k
\end{align}
for some $k<0$ then the comparison inequality implies
\begin{align}
\tfrac{1}{2} \Delta_{\mathcal{H}} r_\epsilon^2 \leq \, &\min \bigg\lbrace 1, \frac{1}{{\nv}_\epsilon} -1\bigg\rbrace + {\nv}_\epsilon\\
&+ \sqrt{\nv_\epsilon |k|} r_\epsilon \coth \sqrt{\nv_\epsilon |k|} r_\epsilon \left( n-2  + \frac{1 - \frac{\tanh \sqrt{\nv_\epsilon |k|} r_\epsilon}{\sqrt{\nv_\epsilon |k|} r_\epsilon}  + \epsilon |k|}{1 - \frac{2\tanh \sqrt{\nv_\epsilon |k|} r_\epsilon /2}{\sqrt{\nv_\epsilon |k|} r_\epsilon}  + \epsilon |k|}\right)\\
\leq \, &\min \bigg\lbrace 1, \frac{1}{{\nv}_\epsilon} -1\bigg\rbrace + {\nv}_\epsilon\\
&+ \sqrt{\nv_\epsilon |k|} r_\epsilon \coth \sqrt{\nv_\epsilon |k|} r_\epsilon \left( n-2  + \frac{1 - \frac{\tanh \sqrt{\nv_\epsilon |k|} r_\epsilon}{\sqrt{\nv_\epsilon |k|} r_\epsilon}}{1 - \frac{2\tanh \sqrt{\nv_\epsilon |k|} r_\epsilon /2}{\sqrt{\nv_\epsilon |k|} r_\epsilon}}\right)\\
\leq \, &\min \bigg\lbrace 1, \frac{1}{{\nv}_\epsilon} -1\bigg\rbrace + {\nv}_\epsilon + n+2 + (n+2)\sqrt{\nv_\epsilon |k|}r_\epsilon. \label{eq:neqsasak}
\end{align}
Therefore, constants $\nu_\epsilon$ and $\lambda_\epsilon$ can easily be chosen so that \eqref{eq:quadbound} is satisfied. Using, for example, the fact that $r_\epsilon \leq \frac{1}{2}\left(\alpha+\frac{r_\epsilon^2}{\alpha}\right)$ for any $\alpha >0$ we see that \eqref{eq:quadbound} is satisfied with
\begin{align}
\nu_\epsilon = n + \tfrac{7}{2} + \tfrac{\alpha}{2}(n+2)\sqrt{|k|}, \quad \lambda_\epsilon = \tfrac{1}{2 \alpha}(n+2)\sqrt{|k|}
\end{align}
for any $\alpha >0$. Note how in this case, constant uniform lower bounds on $\mathbf{K}_{\mathcal{H},J}$ and ${\Ricbf}_{\mathcal{H},J^\perp}$ imply that $\tfrac{1}{2} \Delta_{\mathcal{H}} r_\epsilon^2$ is bounded above by a linear function of $r_\epsilon$, whereas the condition \eqref{eq:quadbound} allows for a quadratic function. Indeed, if the curvatures $\mathbf{K}_{\mathcal{H},J}$ and ${\Ricbf}_{\mathcal{H},J^\perp}$ are bounded below not by a constant, but by a negative quadratic in $r_\epsilon$, then constants $\nu_\epsilon$ and $\lambda_\epsilon$ can still be found such that \eqref{eq:quadbound} is satisfied. Unbounded curvature is thus permitted in this particular setting.

\subsection{It\^{o} formula}

Now suppose $((X_t)_{t \geq 0},(\mathbb{P}_x)_{x\in M})$ is a sub-elliptic diffusion process generated by $\frac{1}{2}\Delta_{\mathcal{H}}$. Note that $X$ admits a smooth heat kernel, by the hypoellipticity of $\Delta_{\mathcal{H}}$. Denote by $\zeta$ the lifetime of $X$. Under the measure $\mathbb{P}_x$ the diffusion satisfies $X_0 = x$. It has recently been proved, in \cite{BGKT2018}, that for each $x \in M$ and $\epsilon >0$ there exists a continuous non-decreasing process $l$ that increases only when $X_t \in \Cut_\epsilon(x_0)$ and a real-valued martingale $\beta$ with
\begin{equation}\label{eq:quadvar}
d \langle \beta \rangle_t = \| \nabla_{\mathcal{H}} r_\epsilon \|^2(X_t) dt,
\end{equation}
such that
\begin{equation}\label{eq:itoform}
r_\epsilon (X_{t\wedge \zeta}) = r_\epsilon (x) + \beta_t + \frac{1}{2}\int_0^{t\wedge \zeta} \Delta_{\mathcal{H}}r_\epsilon (X_s) ds - l_{t\wedge \zeta}
\end{equation}
holds $\mathbb{P}_x$-almost surely. Note that under $\mathbb{P}_x$ the Lebesgue measure of the set of times when $X \in \Cut_\epsilon(x_0)$ is almost surely zero, so the integral in \eqref{eq:itoform} is well defined. Indeed, the distibutional part of $\Delta_{\mathcal{H}}r_\epsilon$ is captured by the geometric local time $l$. The idea now is to combine inequality \eqref{eq:quadbound} with the It\^{o} formula \eqref{eq:itoform} to derive estimates on the even moments of $r_\epsilon(X_t)$. Following the approach laid out in \cite{Thompson2016}, which was similar to that of \cite[Theorem~5.40]{Stroock2000}, these will imply bounds on the moment generating function of $r_\epsilon^2(X_t)$ and consequently a concentration inequality and exit time estimate.

\subsection{Second radial moment}

We begin by calculating an estimate on the second moment of $r_\epsilon(X_t)$, which will then be used as the base case in an induction argument yielding estimates for the higher even moments.

\begin{theorem}\label{th:secradmomthm}
Suppose there exists constants ${\nu_\epsilon} \geq 1$ and ${\lambda_\epsilon} \in \mathbb{R}$ such that inequality \eqref{eq:quadbound} holds. Then $X$ is non-explosive and
\begin{equation}\label{eq:secondmoment}
\mathbb{E}^x \left[ r_\epsilon^{2}(X_t)\right] \leq r_\epsilon^2(x)e^{{\lambda_\epsilon} t} +{\nu_\epsilon} t\Lambda_\epsilon(t)
\end{equation}
where
\begin{equation}
\Lambda_\epsilon(t):= \frac{e^{{\lambda_\epsilon} t}-1}{\lambda_\epsilon t}
\end{equation}
for all $t\geq 0$.
\end{theorem}

\begin{proof}
Let $\lbrace D_i \rbrace_{i=1}^{\infty}$ be an exhaustion of $M$ by regular domains and denote by $\tau_{D_i}$ the first exit time of $X$ from $D_i$. Note that $\tau_{D_i} < \tau_{D_{i+1}}$ and that this sequence of stopping times announces the explosion time $\zeta$. Then, by \eqref{eq:quadvar} and the It\^{o} formula \eqref{eq:itoform}, together with the fact that
\begin{equation}\label{eq:deltaprodform}
\tfrac{1}{2} \Delta_{\mathcal{H}} r_\epsilon^2 = r_\epsilon \Delta_{\mathcal{H}} r_\epsilon + \| \nabla_{\mathcal{H}} r_\epsilon \|^2,
\end{equation}
it follows that
\begin{equation}\label{eq:itoforsecond}
r_\epsilon^2 (X_{t\wedge \tau_{D_i}}) =r_\epsilon^2 (x) + 2 \int_{0}^{t \wedge \tau_{D_i}} r_\epsilon (X_s) \left( d\beta_s -dl_s\right) + \frac{1}{2}\int_{0}^{t \wedge \tau_{D_i}} \Delta_{\mathcal{H}} r_\epsilon^2 (X_s) ds
\end{equation}
holds, $\mathbb{P}_x$-almost surely. Since the domains $D_i$ are of compact closure the It\^{o} integral in \eqref{eq:itoforsecond} is a martingale and so
\begin{equation}
\begin{split}
\mathbb{E}^x\left[r_\epsilon^2(X_{t\wedge \tau_{D_i}})\right] = r_\epsilon^2(x) &- 2\,\mathbb{E}^x\left[ \int_0^{t\wedge \tau_{D_i}} r_\epsilon(X_s)dl_s\right] +  \frac{1}{2}\int_{0}^{t} \mathbb{E}^x\left[\mathbf{1}_{\lbrace s < \tau_{D_i}\rbrace}\Delta_{\mathcal{H}} r_\epsilon^2 (X_s) \right] ds
\end{split}
\end{equation}
for all $t\geq 0$. Before applying Gronwall's inequality we should be careful, since we are allowing the coefficient ${\lambda_\epsilon}$ to be negative. For this, note that
\begin{equation}\
\mathbb{E}^x\left[r_\epsilon^2(X_{t\wedge \tau_{D_i}})\right] = \mathbb{E}^x\left[\mathbf{1}_{\lbrace t < \tau_{D_i}\rbrace} r_\epsilon^2(X_t)\right]+\mathbb{E}^x\left[\mathbf{1}_{\lbrace t \geq \tau_{D_i}\rbrace} r_\epsilon^2(X_{\tau_{D_i}})\right]
\end{equation}
and that the two functions
\begin{equation}
t\mapsto \mathbb{E}^x\left[ \int_0^{t\wedge \tau_{D_i}} r_\epsilon(X_s)dl_s\right], \quad t\mapsto \mathbb{E}^x\left[\mathbf{1}_{\lbrace t \geq \tau_{D_i}\rbrace}r_\epsilon^2(X_{\tau_{D_i}})\right]
\end{equation}
are nondecreasing, so if we define a function $f_{x,i,2}$ by
\begin{equation}
f_{x,i,2}(t):= \mathbb{E}^x\left[ \mathbf{1}_{\lbrace t<\tau_{D_i}\rbrace} r_\epsilon^2(X_t)\right]
\end{equation}
then $f_{x,i,2}$ is differentiable and we have the differential inequality
\begin{equation}\label{eq:diffineq2}\
\begin{cases}
f'_{x,i,2}(t)\leq {\nu_\epsilon} + {\lambda_\epsilon} f_{x,i,2}(t)\\
f_{x,i,2}(0)=r_\epsilon^2(x)
\end{cases}
\end{equation}
for all $t\geq 0$. Applying Gronwall's inequality to \eqref{eq:diffineq2} yields
\begin{equation}\label{eq:basecase}
\mathbb{E}^x \left[ \mathbf{1}_{\lbrace t<\tau_{D_i}\rbrace}r_\epsilon^2(X_t) \right] \leq r_\epsilon^2 (x)e^{{\lambda_\epsilon} t} +{\nu_\epsilon}\left(\frac{e^{{\lambda_\epsilon} t}-1}{{\lambda_\epsilon}}\right)
\end{equation}
for all $t\geq0$. Choosing $D_i = B^\epsilon_i(x_0) := \lbrace y \in M : r_\epsilon(y) < i\rbrace$, inequality \eqref{eq:basecase} implies
\begin{equation}
\mathbb{P}^x \lbrace \tau_{B^\epsilon_i(x_0)} \leq t \rbrace  \leq \frac{r_\epsilon^2(x)e^{\lambda_\epsilon t}+{\nu_\epsilon} t \Lambda_\epsilon (t)}{i^2} \nonumber
\end{equation}
for all $t\geq0$, which implies that $X$ is non-explosive. Inequality \eqref{eq:secondmoment} therefore follows from \eqref{eq:basecase} by the monotone convergence theorem.
\end{proof}

We will refer the object on the left-hand side of inequality \eqref{eq:secondmoment} as \textit{the second radial moment of $X_t$ with respect to $x_0$}. To find an inequality for \textit{the first radial moment of $X_t$ with respect to $x_0$} one can simply use Jensen's inequality. Note that $\lim_{{\lambda_\epsilon} \rightarrow 0} \Lambda_\epsilon(t)= t$, which provides the sense in which Theorem \ref{th:secradmomthm} and similar statements should be interpreted when ${\lambda_\epsilon} =0$.

\subsection{Higher even radial moments}\label{ss:hrm}

Recall that if $Y$ is a real-valued Gaussian random variable with mean $\mu$ and variance $\sigma^2$ then for $\mom \in \mathbb{N}$ one has the formula
\begin{equation}\label{eq:absolutelag}
\mathbb{E}\left[  Y^{2\mom} \right] = \left(2\sigma^2\right)^\mom \mom! L^{-\frac{1}{2}}_\mom \left(-\frac{\mu^2}{2\sigma^2}\right)
\end{equation}
where $L^{\alpha}_{\mom}(z)$ are the \textit{Laguerre polynomials}, defined by the formula
\begin{equation}
L^{\alpha}_{\mom}(z) = e^z \frac{z^{-\alpha}}{\mom!}\frac{\partial^\mom}{\partial z^\mom}\left(e^{-z}z^{\mom+\alpha}\right)
\end{equation}
for $\mom=0,1,2,\ldots$ and $\alpha >-1$. For all the properties of Laguerre polynomials used in this article, see \cite{Lebedev1972}. In particular, if $X$ is a standard Brownian motion on $\mathbb{R}$ then
\begin{equation}\label{eq:brownianmoments}
\mathbb{E}^x \left[\vert X_t \vert^{2\mom} \right]= \left(2t\right)^\mom \mom! L^{-\frac{1}{2}}_\mom \left(-\frac{\vert x\vert^2}{2t}\right)
\end{equation}
for all $t\geq0$. With this in mind we prove the following theorem:

\begin{theorem}\label{th:evenradmomthm}
Suppose there exist constants ${\nu_\epsilon} \geq1$ and ${\lambda_\epsilon} \in \mathbb{R}$ such that inequality \eqref{eq:quadbound} holds and let $\mom \in \mathbb{N}$. Then
\begin{equation}\label{eq:evenmom}
\mathbb{E}^x\left[r_\epsilon^{2\mom}(X_t)\right] \leq  \left(2t \Lambda_\epsilon(t)\right)^\mom \mom! L^{\frac{{\nu_\epsilon}}{2}-1}_\mom \left(-\frac{r_\epsilon^2(x)e^{\lambda_\epsilon t}}{2t{\Lambda_\epsilon}(t)}\right)
\end{equation}
for all $t\geq 0$, where ${\Lambda_\epsilon}(t)$ is defined as in Theorem \ref{th:secradmomthm}.
\end{theorem}

\begin{proof}
By \eqref{eq:quadbound} it follows that, on $M \setminus \Cut_\epsilon(x_0)$ and for $\mom \in \mathbb{N}$, we have
\begin{equation}
\tfrac{1}{2}\Delta_{\mathcal{H}} r_\epsilon^{2\mom} \leq \mom\left({\nu_\epsilon} + 2\left(\mom-1\right)\right)r_\epsilon^{2\mom-2} + \mom {\lambda_\epsilon} r_\epsilon^{2\mom},
\end{equation}
and by the It\^{o} formula \eqref{eq:itoform}, using \eqref{eq:quadvar} and \eqref{eq:deltaprodform}, we have
\begin{equation}
r_\epsilon^{2\mom} (X_{t\wedge \tau_{D_i}}) = r_\epsilon^{2\mom} (x) + 2\mom \int_{0}^{t \wedge \tau_{D_i}} r_\epsilon^{2\mom-1} (X_s) \left(d\beta_s- dl_s\right) + \frac{1}{2}\int_{0}^{t \wedge \tau_{D_i}} \Delta_{\mathcal{H}} r_\epsilon^{2\mom} (X_s) ds
\end{equation}
for all $t\geq 0$, almost surely, where the stopping times $\tau_{D_i}$ are defined as in the proof of Theorem \ref{th:secradmomthm}. It follows that if we define functions $f_{x,i,2\mom}$ by
\begin{equation}
f_{x,i,2\mom}(t):= \mathbb{E}^x\left[\mathbf{1}_{\lbrace t<\tau_{D_i}\rbrace} r_\epsilon^{2\mom}(X_t)\right]
\end{equation}
then, arguing as we did in the proof of Theorem \ref{th:secradmomthm}, we have the differential inequalities
\begin{equation}\label{eq:togron}
\begin{cases}
f'_{x,i,2\mom}(t)\leq \mom\left({\nu_\epsilon} + 2\left(\mom-1\right)\right) f_{x,i,2(\mom-1)}(t) + \mom{\lambda_\epsilon} f_{x,i,2\mom}(t)\\
f_{x,i,2\mom}(0)=r_\epsilon^{2\mom}(x)
\end{cases}
\end{equation}
for all $t\geq 0$. Applying Gronwall's inequality yields
\begin{equation}\label{eq:evenmoments}
f_{x,i,2\mom}(t) \leq \left(r_\epsilon^{2\mom} (x) +  \mom\left({\nu_\epsilon} + 2\left(\mom-1\right)\right) \int_0^t f_{x,i,2(\mom-1)}(s) e^{-\mom{\lambda_\epsilon} s} ds\right)e^{\mom{\lambda_\epsilon} t}
\end{equation}
for all $t \geq 0$ and $\mom \in \mathbb{N}$. The next step in the proof is to use induction on $\mom$ to show that
\begin{equation}\label{eq:inductionineq}
f_{x,i,2\mom}(t) \leq \sum_{\sv=0}^{\mom} \binom{\mom}{\sv} \left(2{\lambda_\epsilon}(t)\right)^{\mom-\sv} r^{2\sv}_\epsilon (x) \frac{\Gamma (\frac{{\nu_\epsilon}}{2}+\mom)}{\Gamma(\frac{{\nu_\epsilon}}{2}+\sv)}e^{\mom{\lambda_\epsilon} t} 
\end{equation}
for all $t \geq 0$ and $\mom \in \mathbb{N}$, where ${\lambda_\epsilon}(t):=(1-e^{-{\lambda_\epsilon} t})/{\lambda_\epsilon}$. Inequality \eqref{eq:basecase} covers the base case $\mom=1$. If we hypothesise that the inequality holds for some $\mom-1$ then by inequality \eqref{eq:evenmoments} we have
\begin{equation}\label{eq:prelimind}
f_{x,i,2\mom}(t) \leq \bigg(r_\epsilon^{2\mom} (x) + \mom\left({\nu_\epsilon} + 2\left(\mom-1\right)\right)\sum_{\sv=0}^{\mom-1} \binom{\mom-1}{\sv} r^{2\sv}_\epsilon (x) \frac{\Gamma (\frac{{\nu_\epsilon}}{2}+\mom-1)}{\Gamma(\frac{{\nu_\epsilon}}{2}+\sv)}{\tilde{\lambda}_\epsilon}(t) \bigg)e^{\mom{\lambda_\epsilon} t}
\end{equation}
for all $t\geq 0$, where ${\tilde{\lambda}_\epsilon}(t)= \int_0^t\left(2{\lambda_\epsilon}(s)\right)^{\mom-1-\sv}e^{-{\lambda_\epsilon} s} ds$. Using $2(\mom-\sv){\tilde{\lambda}_\epsilon}(t) = (2{\lambda_\epsilon}(t))^{\mom-\sv}$ and properties of the Gamma function it is straightforward to deduce inequality \eqref{eq:inductionineq} from inequality \eqref{eq:prelimind}, which completes the inductive argument. Since ${\nu_\epsilon} \geq 1$ we can then apply the relation 
\begin{equation}\label{eq:lagpropone}
L^{\alpha}_\mom (z) = \sum_{\sv=0}^{\mom} \frac{\Gamma(\mom+\alpha+1)}{\Gamma(\sv+\alpha+1)}\frac{(-z)^\sv}{\sv!(\mom-\sv)!},
\end{equation}
which can be proved using Leibniz's formula, to see that
\begin{equation}\label{eq:hyper0}
\sum_{\sv=0}^{\mom} \binom{\mom}{\sv} (2{\lambda_\epsilon}(t))^{\mom-\sv} r_\epsilon^{2\sv}(x) \frac{\Gamma (\frac{{\nu_\epsilon}}{2}+\mom)}{\Gamma(\frac{{\nu_\epsilon}}{2}+\sv)} = (2t {\Lambda_\epsilon}(t) e^{\lambda_\epsilon t})^\mom \mom! L^{\frac{{\nu_\epsilon}}{2}-1}_\mom \left(-\frac{r_\epsilon^2(x)e^{\lambda_\epsilon t}}{2t {\Lambda_\epsilon}(t) }\right)
\end{equation}
and so by inequality \eqref{eq:inductionineq} it follows that
\begin{equation}\label{eq:afhyp}
f_{x,i,2\mom}(t) \leq  \left(2t{\Lambda_\epsilon}(t)\right)^\mom \mom! L^{\frac{{\nu_\epsilon}}{2}-1}_\mom \left(-\frac{r_\epsilon^2(x)e^{\lambda_\epsilon t}}{2t {\Lambda_\epsilon}(t) }\right)
\end{equation}
for $t\geq0$ and $i,\mom \in \mathbb{N}$. The result follows from this by the monotone convergence theorem.
\end{proof}

We will refer the object on the left-hand side of inequality \eqref{eq:evenmom} as \textit{the $2\mom$-th radial moment of $X_t$ with respect to $x_0$}. One can deduce an estimate for \textit{the $(2\mom-1)$-th radial moment of $X_t$ with respect to $x_0$} by Jensen's inequality.

\subsection{Exponential estimate}\label{ss:exprm}

For $\vert \gamma \vert<1$ the Laguerre polynomials also satisfy the identity
\begin{equation}\label{eq:lagproptwo}
\sum_{\mom=0}^{\infty} \gamma^\mom L^{\alpha}_\mom(z) = (1-\gamma)^{-(\alpha+1)} e^{-\frac{z\gamma}{1-\gamma}}.
\end{equation}
It follows from this identity and equation \eqref{eq:absolutelag} that for a real-valued Gaussian random variable $Y$ with mean $\mu$ and variance $\sigma^2$ we have for $\theta \geq 0$ that
\begin{equation}
\mathbb{E} \left[e^{\frac{\theta}{2} \vert Y\vert^2}\right] = \left(1-\theta \sigma^2\right)^{-\frac{1}{2}}\exp\left(\frac{\theta\vert \mu \vert ^2 }{2(1-\theta \sigma^2)}\right)
\end{equation}
so long as $\theta \sigma^2 <1$ (and there is a generalization of this for Gaussian measures on Hilbert spaces). In particular, if $X$ is a standard Brownian motion on $\mathbb{R}$ starting from $x \in \mathbb{R}$ then for $t\geq0$ it follows that
\begin{equation}\label{eq:brownianexp}
\mathbb{E} \left[e^{\frac{\theta}{2} \vert X_t(x)\vert^2}\right] = \left(1-\theta t\right)^{-\frac{1}{2}}\exp\left(\frac{\theta \vert x \vert^2 }{2(1-\theta t)}\right)
\end{equation}
so long as $\theta t <1$. With this in mind we prove the following theorem:

\begin{theorem}\label{th:expenbm}
Suppose there exists constants ${\nu_\epsilon} \geq1$ and ${\lambda_\epsilon} \in \mathbb{R}$ such that inequality \eqref{eq:quadbound} holds. Then
\begin{equation}\label{eq:exponential2}
\mathbb{E}^x \left[e^{\frac{\theta}{2} r_\epsilon^2 (X_t)}\right] \leq \left(1-\theta t {\Lambda_\epsilon}(t) \right)^{-\frac{{\nu_\epsilon}}{2}}\exp\left(\frac{\theta r_\epsilon^2(x)e^{{\lambda_\epsilon} t}}{2(1-\theta t {\Lambda_\epsilon}(t) )}\right)
\end{equation}
for all $t,\theta \geq 0$ such that $\theta t {\Lambda_\epsilon}(t) <1$, where ${\Lambda_\epsilon}(t)$ is defined as in Theorem \ref{th:secradmomthm}.
\end{theorem}

\begin{proof}
Using inequality \eqref{eq:afhyp} and equation \eqref{eq:lagproptwo} we see that
\begin{equation}\label{eq:exponential1}
\begin{split}
\text{ }\mathbb{E}^x \left[\mathbf{1}_{\{t< \tau_{D_i} \}} e^{\frac{\theta}{2} r_\epsilon^2 (X_t)}\right] =&\text{ }\sum_{\mom=0}^{\infty} \frac{\theta^\mom}{2^\mom \mom!} f_{x,i,2\mom}(t) \\
\leq&\text{ }\sum_{\mom=0}^{\infty} \left(\theta t{\Lambda_\epsilon}(t)\right)^\mom L^{\frac{{\nu_\epsilon}}{2}-1}_\mom \left(-\frac{r_\epsilon^2(x)e^{\lambda_\epsilon t}}{2t{\Lambda_\epsilon}(t)}\right) \\
=&\text{ } \left(1-\theta t {\Lambda_\epsilon}(t) \right)^{-\frac{{\nu_\epsilon}}{2}}\exp\left(\frac{\theta r_\epsilon^2(x)e^{{\lambda_\epsilon} t}}{2(1-\theta t {\Lambda_\epsilon}(t))}\right) \nonumber
\end{split}
\end{equation}
where we justify switching the order of integration with the stopping time. The result follows by the monotone convergence theorem.
\end{proof}

The following corollary concerns the sub-Riemannian limit as $\epsilon \rightarrow 0$. It will be assumed in this corollary (and in two subsequent theorems) that the constants $\nu_\epsilon$ and $\lambda_\epsilon$ converge as $\epsilon \rightarrow 0$. If the constants have been chosen is such a way that they do not converge, but are nonetheless uniformly bounded in $\epsilon$, then they can simply be replaced with these uniform bounds, so that the convergence then trivially holds. Note that in Example \ref{ex:reeb} we found suitable constants that were indeed chosen independently of $\epsilon$.

\begin{corollary}\label{cor:expenbmr0}
For each $\epsilon >0$ suppose there exists constants ${\nu_\epsilon} \geq1$ and ${\lambda_\epsilon} \in \mathbb{R}$ such that inequality \eqref{eq:quadbound} holds and such that $(\nu_\epsilon,\lambda_\epsilon) \rightarrow (\nu, \lambda )$ as $\epsilon \rightarrow 0$. Then
\begin{equation}\label{eq:exponential2r0}
\mathbb{E}^x \left[e^{\frac{\theta}{2} r_0^2 (X_t)}\right] \leq \left(1-\theta t {\Lambda}(t)\right)^{-\frac{{\nu}}{2}}\exp\left(\frac{\theta r_0^2(x)e^{{\lambda} t}}{2(1-\theta t {\Lambda}(t))}\right)
\end{equation}
for all $t,\theta \geq 0$ such that $\theta t {\Lambda}(t)<1$, where ${\Lambda}(t):=(e^{{\lambda} t} -1)/{\lambda t}$.
\end{corollary}

\begin{proof}
It follows from Theorem \ref{th:expenbm} that for each $t,\theta$ satisfying the conditions of the theorem, the sequence of random variables $$\bigg\lbrace e^{\frac{\theta}{2}r_{1/n}^2(X_t)} \bigg\rbrace_{n \in \mathbb{N}}$$ is uniformly integrable and, therefore, inequality \eqref{eq:exponential2r0} follows from \eqref{eq:exponential2} by setting $\epsilon = 1/n$ and taking the limit $n \rightarrow \infty$ of both sides of \eqref{eq:exponential2}.
\end{proof}

\subsection{Concentration inequality}\label{ss:cncntrtnineq}

If $X$ is a Brownian motion on $\mathbb{R}^{\dimm}$ starting at $x$ then it is easy to see that
\begin{equation}\label{eq:asymrmn3}
\lim_{r\rightarrow \infty}\frac{1}{r^2}\log \mathbb{P}^x\lbrace X_t \not\in B_r(x_0)\rbrace = -\frac{1}{2t}
\end{equation}
for all $t>0$. Note that the right-hand side does not depend on the dimension $\dimm$. With this in mind we prove the following theorem, for which we recall that $B^0_r(x_0)$ denotes the open ball centred at $x_0$ and with radius $r$ in the sub-Riemannian distance $d_0$:

\begin{theorem}\label{th:concenthm}
For each $\epsilon >0$ suppose there exists constants ${\nu_\epsilon} \geq1$ and ${\lambda_\epsilon} \in \mathbb{R}$ such that inequality \eqref{eq:quadbound} holds and such that $(\nu_\epsilon,\lambda_\epsilon) \rightarrow (\nu, \lambda )$ as $\epsilon \rightarrow 0$. Then
\begin{equation}\label{eq:concenasymineq}
\lim_{r\rightarrow \infty} \frac{1}{r^2}\log \mathbb{P}^x \lbrace X_t \notin B^0_r(x_0) \rbrace \leq -\frac{1}{2t \Lambda(t)}
\end{equation}
for all $t>0$, where ${\Lambda}(t)$ is defined as in Corollary \ref{cor:expenbmr0}.
\end{theorem}

\begin{proof}
For $\theta \geq 0$ and $r>0$ it follows from Markov's inequality and Corollary \ref{cor:expenbmr0} that
\begin{equation}
\begin{split}
\mathbb{P}^x \lbrace X_t \notin B^0_r(x_0)\rbrace &= \mathbb{P}^x \lbrace r_0(X_t)\geq r\rbrace\\[2mm]
&= \mathbb{P}^x \lbrace e^{\frac{\theta }{2}r_0^2(X_t)}\geq e^{\frac{\theta }{2}r^2}\rbrace\\[2mm]
&\leq e^{-\frac{\theta }{2}r^2}\mathbb{E}^x\left[e^{\frac{\theta }{2}r_0^2(X_t)}\right] \\
&\leq \left(1-\theta {\lambda}(t) e^{{\lambda} t}\right)^{-\frac{{\nu}}{2}} \exp \left(\frac{\theta r_0^2(x) e^{{\lambda} t}}{2(1-\theta {\lambda}(t) e^{{\lambda} t})}-\frac{\theta r^2}{2}\right)\nonumber
\end{split}
\end{equation}
so long as $\theta {\lambda}(t) e^{{\lambda} t}<1$. If $t>0$ then choosing $\theta =\delta ({\lambda}(t)e^{{\lambda} t})^{-1} $ shows that for any $\delta \in \left[0,1\right)$ and $r>0$ we have the estimate
\begin{equation}\label{eq:concenineq}
\mathbb{P}^x \lbrace X_t \notin B^0_r(x_0) \rbrace \leq (1-\delta)^{-\frac{{\nu}}{2}}\exp\left(\frac{r_0^2(x)\delta e^{\lambda t}}{2(1-\delta)t{\Lambda}(t)}-\frac{\delta r^2}{2t\Lambda(t)}\right)
\end{equation}
from which the theorem follows, since $\delta$ can be chosen arbitrarily close to $1$ after taking the limit.
\end{proof}

\subsection{Exit time estimate}

For the case $\lambda_\epsilon \geq 0$, here is the exit time estimate:

\begin{theorem}\label{thm:exittime}
Fix $r>0$ and for each $\epsilon >0$ suppose ${\nu_\epsilon}\geq 1$ and ${\lambda_\epsilon} \geq 0$ are constants such that the inequality \eqref{eq:quadbound} holds on the ball $B^\epsilon_r(x_0)$ with $(\nu_\epsilon,\lambda_\epsilon) \rightarrow (\nu, \lambda )$ as $\epsilon \rightarrow 0$. Then
\begin{equation}
\mathbb{P}^x\bigg\lbrace \sup_{s\in \left[0,t\right]} r_0(X_s) \geq r \bigg\rbrace\leq (1-\delta)^{-\frac{{\nu}}{2}}\exp\left(\frac{r_0^2(x)\delta e^{\lambda t}}{2(1-\delta){t \Lambda}(t)}-\frac{\delta r^2}{2t\Lambda(t)}\right)
\end{equation}
for all $t>0$ and $\delta \in (0,1)$.
\end{theorem}

\begin{proof}
The proof requires a slight modification of the argument we used to derive Theorems \ref{th:expenbm} and \ref{th:concenthm}. In particular, denoting by $\tau_r$ the first exit time of $X$ from the ball and applying the It\^{o} formula as in Theorem \ref{th:evenradmomthm}, we use the assumption $\lambda_\epsilon \geq 0$ to obtain the slightly different estimate
\begin{equation}
\begin{split}
\E^x\left[r_\epsilon^{2\mom} (X_{t\wedge \tau_{r}})\right] \leq \, r_\epsilon^{2\mom} (x) &+ \frac{\mom}{2} \left({\nu_\epsilon} + 2\left(\mom-1\right)\right)\int_{0}^{t } \E^x\left[r_\epsilon^{2\mom-2}(X_{s\wedge \tau_{r}})\right]ds\\
& +\frac{\mom {\lambda_\epsilon}}{2}  \int_{0}^{t} \E^x\left[r_\epsilon^{2\mom}(X_{s\wedge \tau_{r}})\right] ds
\end{split}
\end{equation}
which, following the inductive argument of earlier, yields moment estimates which can then be summed, as in Theorem \ref{th:expenbm}, to obtain an exponential estimate for the stopped process. Taking the limit as $\epsilon \downarrow 0$, it follows that the right-hand side of the inequality
\begin{equation}
e^{\frac{\theta}{2}r^2}  \mathbb{P}^x\bigg\lbrace \sup_{s\in \left[0,t\right]} r_0(X_s) \geq r \bigg\rbrace \leq \E^x\left[e^{\frac{\theta}{2}r_0^2(X_{t\wedge \tau_r})}\right]
\end{equation}
is bounded by the right-hand side of \eqref{eq:exponential2r0}. Choosing $\theta$ as in the proof of Theorem \ref{th:concenthm} yields the desired estimate.
\end{proof}

\section{Metric measure spaces}\label{sec:mms}

We next consider the setting of an $\mathrm{RCD}^*(K, N)$ space $(X, d, \mathfrak{m})$, meaning a geodesic metric measure space having a notion of a lower Ricci curvature bound $K \in \mathbb{R}$ together with a notion of an upper bound $N \in [1,\infty)$ on dimension. We will give a concise introduction to this setting based on that of \cite{KK2018}, whose It\^{o} formula we will describe in the next subsection.

We start with a metric measure space $(X,d,\mathfrak{m})$, meaning $(X,d)$ is a complete, separable metric space and $\mathfrak{m}$ a $\sigma$-finite Borel measure on $X$. Suppose $\mathfrak{m}(B_r(x)) \in (0,\infty)$ for any metric ball $B_r(x)$ of radius $r>0$ centred at $x \in X$. Suppose $d$ is a geodesic distance, meaning that for any $x_0,x_1 \in X$ there exists $\gamma :[0,1]\rightarrow X$ such that $\gamma(0) = x_0,\gamma(1)=x_1$ and $d(\gamma(s),\gamma(t)) = |s-t|d(x_0,x_1)$. Denote by $\mathcal{C}^{\textrm{Lip}}(X)$ the Lipschitz functions on $X$ and define Cheeger’s energy functional $\mathrm{Ch}: L^2(X;\mathfrak{m}) \rightarrow [0,\infty]$ by
\begin{equation}
\mathrm{Ch}(f) := \frac{1}{2}\inf\bigg\lbrace \liminf_{n\rightarrow \infty} \int_X |Df_n|^2 d\mathfrak{m} : f_n \in \mathcal{C}^{\textrm{Lip}}(X) \cap L^2(X;\mathfrak{m}),\, f_n \xrightarrow{L^2(X,\mathfrak{m})} f \bigg\rbrace
\end{equation}
with domain given by the Sobolev space
\begin{equation}
\mathcal{D}(\mathrm{Ch}) := \lbrace f \in L^2(X;\mathfrak{m}) : \mathrm{Ch}(f) < \infty \rbrace
\end{equation}
where $|Dg|:X\rightarrow [0,\infty]$ is the local Lipschitz constant of $g :X \rightarrow \mathbb{R}$ defined by
\begin{equation}
|Dg|(x) := \limsup_{y \rightarrow x} \frac{|g(x) - g(y)|}{d(x,y)}.
\end{equation}
For $f \in L^2(X;\mathfrak{m})$ with $\mathrm{Ch}(f) < \infty$ there exists some $|Df|_w \in L^2(X;\mathfrak{m})$ such that
\begin{equation}
\mathrm{Ch}(f) = \frac{1}{2}\int_X |Df|_w^2 \,d\mathfrak{m}
\end{equation}
and we call $|Df|_w$ the \emph{minimal weak upper gradient} of $f$. We call $(X,d,\mathfrak{m})$ \textit{infinitesimally Hilbertian} if $\textrm{Ch}$ satisfies the parallelogram law, in which case the minimal weak upper gradient also satisfies the parallelogram law and there exists a bilinear form
\begin{equation}
\langle D \cdot,D\cdot\rangle : \mathcal{D}(\mathrm{Ch}) \times \mathcal{D}(\mathrm{Ch}) \rightarrow L^1(X;\mathfrak{m})
\end{equation}
such that $\langle D f,Df\rangle = |Df|^2_w$. We denote by $\Delta$ the (non-positive definite) self-adjoint operator associated to $2\textrm{Ch}$, with domain
\begin{equation}
\mathcal{D}(\Delta) := \bigg\lbrace f \in \mathcal{D}(\mathrm{Ch}) : \exists h \in L^2(X,\mathfrak{m}),\, 2\mathrm{Ch}(f,g) = - \int_X hg\, d \mathfrak{m},\, \forall g \in \mathcal{D}(\mathrm{Ch})\bigg\rbrace
\end{equation}
with $\Delta f := h$ for any $f \in \mathcal{D}(\Delta)$.

\begin{definition}\label{def:RCD}
Suppose $K \in \mathbb{R}$ and $N \in [1,\infty)$. We say that $(X,d,\mathfrak{m})$ is an $\mathrm{RCD}^*(K,N)$ space if it satisfies the following conditions:
\begin{enumerate}
\item[i)] It is infinitesimally Hilbertian;
\item[ii)] There exists $x_0\in X$ and constants $c_1,c_2>0$ such that $\mathfrak{m}(B_r(x_0)) \leq c_1 e^{c_2 r^2}$ for all $r>0$;
\item[iii)] Any $f \in \mathcal{D}(\mathrm{Ch})$ satisfying $|Df|_w\leq 1$ $\mathfrak{m}$-a.e. has a $1$-Lipschitz representative;
\item[iv)] For any $f \in \mathcal{D}(\Delta)$ with $\Delta f \in \mathcal{D}(\mathrm{Ch})$ and $g \in \mathcal{D}(\Delta) \cap L^\infty(X;\mathfrak{m})$ with $g \geq 0$ and $\Delta g \in L^\infty(X;\mathfrak{m})$, there is the weak Bochner inequality
\begin{equation}
\frac{1}{2}\int_X |Df|^2 \Delta g \,d\mathfrak{m} - \int_X \langle Df,D\Delta f\rangle g \,d\mathfrak{m} \geq K\int_X |Df|^2_w g \,d\mathfrak{m} + \frac{1}{N}\int_X |\Delta f|^2 g \,d \mathfrak{m}.
\end{equation}
\end{enumerate}
\end{definition}

In such a setting $(\mathrm{Ch},\mathcal{D}(\mathrm{Ch}))$ is a strongly local regular Dirichlet form, which by $\textrm{ii)}$ and \cite[Theorem~4.20]{AGS2014} is conservative. From now on we fix an $\mathrm{RCD}^*(K,N)$ space $(X,d,\mathfrak{m})$.

\subsection{It\^{o} formula}

Let $((X_t)_{t \geq 0}, (\mathbb{P}_x)_{x\in X})$ be the diffusion process canonically associated with $(\mathrm{Ch},\mathcal{D}(\mathrm{Ch}))$. Note that for a bounded measurable function $f$, Fukushima's theory of Dirichlet forms implies the correspondence $P_tf(x) = \E^x[f(X_t)]$ only for $\mathrm{Ch}$-a.e. $x \in X$, where by $P_t$ we mean the semigroup generated by $\frac{1}{2}\Delta$, defined a priori via spectral theory. However it has recently been shown in \cite{Guneysu2019}, using the assumption $N<\infty$, that $P_t$ maps $L^\infty(X)$ to $C(X)$ and that the formula actually holds pointwise, for each $x \in X$. In particular, and as in the previous section, by $\mathbb{P}_x$ we mean the probability measure on the space of continuous paths which has as its transition density the jointly continuous kernel of the semigroup $P_t$.

Now fix $x_0 \in X$ and define $r(x):=d(x_0,x)$, as before. Kuwada and Kuwae recently proved an It\^{o} formula for the radial part of the diffusion. For this, they set
\begin{equation}
k :=
\begin{cases}
\frac{K}{N-1}, &N>1 \\
0, &N=1
\end{cases}
\end{equation}
and define
\begin{equation}
\cot_{k}(r) := \frac{\phi'_{k}(r)}{\phi_{k}(r)}
\end{equation}
with $\phi_k(r)$ defined as in Example \ref{ex:reeb}. Then, as \cite[Corollary~5.5]{KK2018}, they prove that there exists a standard one-dimensional Brownian motion $B$ and a positive continuous additive functional $A$ such that, for all suitable functions $f \in C^2(\mathbb{R})$, there is the formula
\begin{align}
f(r(X_t)) =\,& f(r(X_0)) + \int_0^t f'(r(X_s))dB_s + \frac{1}{2}\int_0^t f''(r(X_s))ds \\
&+ \frac{N-1}{2}\int_0^t f'(r(X_s))\cot_{k}(r(X_s))ds - \frac{1}{2}\int_0^t f'(r(X_s))dA_s\label{eq:itomms}
\end{align}
for all $t \geq 0$, $\mathbb{P}_x$-almost surely for all $x \in X$. This formula is, in particular, valid for all functions $f$ of the form $f(r) = r^{2p}$ with $p\geq 1$.

\subsection{Exponential and exit time estimates}

Using formula \eqref{eq:itomms} we can derive various estimates of precisely the form considered in the previous section, including the second radial moments and other higher even moments. To avoid extensive repetition, let us skip straight to the corresponding exponential estimate, in which we denote by $\tau_D$ the first exit time of $X_t$ from a compact set $D$.

\begin{theorem}\label{th:expenbmII}
Set $\lambda := \frac{1}{2}\sqrt{(N-1)K^-}$. Then
\begin{equation}\label{eq:exponential3}
\mathbb{E}^x \left[e^{\frac{\theta}{2} r^2 (X_{t\wedge \tau_D})}\right] \leq \left(1-\theta t\Lambda(t)\right)^{-\frac{{N + \lambda}}{2}}\exp\left(\frac{\theta r^2(x)e^{{\lambda} t}}{2(1-\theta t\Lambda(t))}\right)
\end{equation}
for all $t,\theta \geq 0$ such that $\theta t \Lambda(t) <1$, where $\Lambda(t):=(e^{{\lambda} t}-1)/\lambda t$.
\end{theorem}

\begin{proof}
We see that
\begin{equation}
r \cot_{k} r \leq 1 + \sqrt{k^-} r \leq 1 +  \frac{\sqrt{k^-}}{2}\left(1+r^2\right).
\end{equation}
By formula \eqref{eq:itomms} applied to the function $f(r) = r^{2p}$ we obtain therefore
\begin{align}
r^{2p}(X_{t\wedge \tau_D}) =\,& r^{2p}(X_0) + 2p \int_0^{t\wedge \tau_D} r^{2p-1}(X_s)dB_s + p(2p-1)\int_0^{t\wedge \tau_D} r^{2p-2}(X_s)ds \\
&+ p(N-1)\int_0^{t\wedge \tau_D} r^{2p-1}(X_s)\cot_{k}(r(X_s))ds - p\int_0^{t\wedge \tau_D} r^{2p-1}(X_s)dA_s \\
\leq\,& r^{2p}(X_0) + 2p \int_0^{t\wedge \tau_D} r^{2p-1}(X_s)dB_s \\
& + p\left(N + \lambda+2(p-1)\right)\int_0^t r^{2p-2}(X_{s\wedge \tau_D})ds + p\lambda \int_0^t r^{2p}(X_{s\wedge \tau_D})ds
\end{align}
for all $t \geq 0$. We can now proceed, as in the proofs of Theorem \ref{th:secradmomthm} and \ref{th:evenradmomthm}, to obtain estimates on the even radial moments which can then be summed, as in the proof of Theorem \ref{th:expenbm}, to obtain the claimed inequality.
\end{proof}

As a corollary, we obtain an analogue of the concentration inequality Theorem \ref{th:concenthm}:

\begin{theorem}\label{th:concenthmmms}
Set $\lambda := \frac{1}{2}\sqrt{(N-1)K^-}$. Then
\begin{equation}
\lim_{r\rightarrow \infty} \frac{1}{r^2}\log \mathbb{P}^x \lbrace X_t \notin B_r(x_0) \rbrace \leq -\frac{1}{2t \Lambda(t)}
\end{equation}
for all $t>0$, where $\Lambda(t):=(e^{{\lambda} t}-1)/\lambda t$.
\end{theorem}

Furthermore we obtain, by Markov's inequality, the following exit time estimate:

\begin{theorem}\label{thm:etemmm}
Set $\lambda := \frac{1}{2}\sqrt{(N-1)K^-}$. Then
\begin{equation}
\mathbb{P}^x\bigg\lbrace \sup_{s\in \left[0,t\right]} r(X_s) \geq r \bigg\rbrace\leq (1-\delta)^{-\frac{{N+\lambda}}{2}}\exp\left(\frac{r^2(x)\delta e^{\lambda t}}{2(1-\delta)t\Lambda(t)}-\frac{\delta r^2}{2t\Lambda(t)}\right)
\end{equation}
for all $t>0$ and $\delta \in (0,1)$.
\end{theorem}

Using this estimate, we can deduce an upper bound for the law of the iterated logarithm. Laws of the iterated logarithm have been recently proved by Kim, Kumagai and Wang in \cite{KimKumagaiWang2017} for the general setting of a metric measure space, assuming a volume doubling condition and suitable bounds on the heat kernel. We assume the $\mathrm{RCD}^*(K,N)$ condition with $K\geq 0$. In this case, we have the following upper bound for the law of iterated logarithm, which follows either from the radial comparison \cite[Theorem~6.1]{KK2018} or from Theorem \ref{thm:etemmm} and the argument given below:

\begin{corollary}
Suppose $K\geq 0$. Then
\begin{equation}
\limsup_{t\rightarrow \infty} \frac{d(X_t,x_0)}{\sqrt{2t\log\log t}} \leq 1
\end{equation}
$\mathbb{P}^{x_0}$-almost surely.
\end{corollary}

\begin{proof}
The argument is taken from \cite[Chapter~5]{MortersPeres2010}. Indeed, fix $\epsilon > 0$, $q >1$ and set
\begin{equation}
A_n := \bigg\lbrace \sup_{s\in \left[0,t\right]} r(X_s) \geq (1+\epsilon) \psi(q^n)\bigg\rbrace
\end{equation}
where $\psi(t) := \sqrt{2t\log \log t}$. Then, by Theorem \ref{thm:etemmm}, for all $0\leq \delta <1$ we have
\begin{align}
\mathbb{P}^{x_0}\lbrace A_n \rbrace \leq (1-\delta)^{-\frac{N}{2}} (n \log q)^{-\delta(1+\epsilon)^2}.
\end{align}
Choosing $\delta$ sufficiently small, for example $\delta = 1/(1+\epsilon)$, we have $\sum_{n=1}^\infty \mathbb{P}^{x_0}\lbrace A_n \rbrace < \infty$ and therefore, by the Borel-Cantelli lemma, only finitely many $A_n$ occur. Consequently, for large $t$, we can write $q^{n-1} \leq t < q^n$ and estimate
\begin{equation}
\frac{d(X_t,x_0)}{\psi(t)} = \frac{d(X_t,x_0)}{\psi(q^n)}\frac{\psi(q^n)}{q^n}\frac{t}{\psi(t)}\frac{q^n}{t} \leq (1+\epsilon)q,
\end{equation}
since $\psi(t)/t$ is decreasing in $t$. The result follows, since $\epsilon>0$ and $q>1$ are arbitrary.
\end{proof}

We can also apply the exit time estimate Theorem \ref {thm:etemmm} to derive upper bounds for Schr\"{o}dinger eigenfunctions, along the lines of \cite{Carmona1978}. We do so in Subsection \ref{ss:upperbnds}. First we must calculate upper bounds on Feynman-Kac functionals with potentials belonging to some $L^p$ space.

\subsection{Feynman-Kac functionals}\label{ss:Feynman}

Sturm's work on the general theory of Dirichlet forms implies the existence of a locally H\"{o}lder continuous representative $p$ on $\supp \m \times \supp \m \times (0,\infty)$ of the heat kernel on $(X, d, \m)$ associated to $\tfrac{1}{2}\Delta$. See \cite[Proposition 2.3]{Sturm1995} and \cite[Corollary 3.3]{Sturm1996}. In this subsection we start with the following assumption:
\begin{itemize}
\item[(\textbf{A1})] There exists $1 \leq n < \infty$, $T_0>0$ and positive constants $c_3 := c_3(T_0)$ and $c_4:=c_4(T_0)$, possibly depending on $T_0$, such that
\begin{equation}
p_t(x,y)\leq c_3 t^{-\frac{n}{2}} \exp\left(-\frac{d^2(x,y)}{c_4 t}\right)
\end{equation}
for all $x,y \in \supp \m$ and $t\in (0,T_0]$.
\end{itemize}
Sharp Gaussian estimates on the heat kernel have been proved by Jiang, Li and Zhang \cite[Theorem 1.2]{JiangLiZhang2016}. In particular, their upper bound states that for any $\epsilon >0$ there exist $C_i := C_i(\epsilon,K,N)>1$ for $i=1,2$, depending only on $K,N$ and $\epsilon$, such that
\begin{equation}
p_t(x,y) \leq \frac{C_1}{\m(B_{\sqrt{t}}(x))} \exp\left(-\frac{d^2(x,y)}{(2+\epsilon)t} + C_2t\right)
\end{equation}
for all $x,y \in \supp \m$ and any $t>0$. Consequently, if $\m$ is lower $n$-Ahlfors regular for some $1\leq n <\infty$, meaning that there exists a constant $c_5>0$ such that
\begin{equation}
\m (B_r(x)) \geq c_5 r^n
\end{equation}
for any $0<r<D$ and for all $x \in X$, where $D$ denotes the diameter of $X$, then assumption (\textbf{A1}) is satisfied. In terms of the constant $n$ given by assumption (\textbf{A1}), we will additionally assume:
\begin{itemize}
\item[(\textbf{A2})] There exist constants $c_1,c_2>0$ such that
\begin{equation}
\m(B_r(x)) \leq c_1 r^n e^{c_2 r^2}
\end{equation}
for all $r >0$ and all $x \in \supp \m$.
\end{itemize}
This assumption replaces the second property in Definition \ref{def:RCD}. For an $\mathrm{RCD}^*(K, N)$ space satisfying assumptions (\textbf{A1}) and (\textbf{A2}), we can deduce various integral estimates for the heat kernel. We begin with the following lemma, which holds true on any metric measure space:

\begin{lemma}\label{lem:polarcoord}
For any $c >0$ we have
\begin{equation}
\int_X e^{-c d^2(x,y)} d\m(y) = 2c \int_0^\infty \m(B_r(x)) r e^{-cr^2} dr
\end{equation}
for all $x \in X$.
\end{lemma}

\begin{proof}
By Fubini's theorem, we have
\begin{equation}
\int_X e^{-c d^2(x,y)} d\m(y) = \int_0^1 \m\left(\Big\lbrace y : e^{-cd^2(x,y)}>\lambda\Big\rbrace\right) d\lambda= \int_0^1 \m(B_{\phi(\lambda)}(x))d\lambda
\end{equation}
where $\phi(\lambda) := \sqrt{-\tfrac{1}{c}\log \lambda}$ for $\lambda \in (0,1)$. Since
\begin{equation}
\phi'(\lambda) = -\frac{1}{2c \phi(\lambda)}e^{c\phi^2(\lambda)}
\end{equation}
it follows, by the change of variables $r = \phi(\lambda)$, that
\begin{equation}
\int_0^1 \m(B_{\phi(\lambda)}(x))d\lambda = 2c\int_0^\infty \m(B_r(x)) r e^{-cr^2} dr
\end{equation}
as claimed.
\end{proof}

For $T>0$ with $q,q'>1$ and a non-negative measurable function $f:[0,T] \times X \rightarrow \R$ define
\begin{equation}
\|f\|_{L^q_{q'} ([0,t] \times X)} := \left( \int_0^t \|f_s(\cdot)\|^{q'}_{L^q(X)} ds\right)^{1/q'}
\end{equation}
for $t \in (0,T]$.

\begin{proposition}\label{prop:lqhk}
Suppose $q,q' >1$. Assume (\textbf{A1}) and suppose there exists $x_0 \in \supp \m$ and constants $c_1,c_2>0$ such that
\begin{equation}
\m(B_r(x_0)) \leq c_1 r^n e^{c_2 r^2}
\end{equation}
for all $r >0$. Choose $T_1>0$ sufficiently small so that
\begin{equation}\label{eq:cond1}
T_1 < T_0 \wedge \frac{q}{c_4 c_2}
\end{equation}
where $T_0$ is the constant determined by assumption (\textbf{A1}). Then there exists a positive constant $C_1:= C_1(q,c_1,c_2,c_3,c_4,n,T_1)$ such that
\begin{equation}
\| p_t(x_0,\cdot)\|_{L^q(X)} \leq C_1 t^{\frac{n}{2}(\frac{1}{q}-1)} \label{eq:lemfirst}
\end{equation}
for all $t \in (0,T_1]$. If in addition
\begin{equation}\label{eq:cond2}
\frac{1}{q} + \frac{2}{nq'}>1
\end{equation}
then furthermore there exists a positive constant $C_2:= C_2(q,q',c_1,c_2,c_3,c_4,n,T_1)$ such that
\begin{equation}
\|p.(x_0,\cdot)\|_{L^q_{q'}([0,t]\times X)} \leq C_2 t^{\frac{n}{2}(\frac{1}{q}-1)+\frac{1}{q'}} \label{eq:lemsecond}
\end{equation}
for all $t \in (0,T_1]$, in which case
\begin{equation}
\lim_{t \downarrow 0} \|p.(x_0,\cdot)\|_{L^q_{q'}([0,t]\times X)} = 0.
\end{equation}
Note that it is always possible to find a $T_1$ satisfying condition \eqref{eq:cond1}.
\end{proposition}

\begin{proof}
For $s \in (0,T_1]$ it follows from Lemma \ref{lem:polarcoord} that
\begin{align}
\| p_s(x_0,\cdot)\|^q_{L^q(X)} \leq \,& c_3^q s^{-\frac{nq}{2}}  \int_X \exp\left(-\frac{q d^2(x_0,y)}{c_4 s}\right)d\m(y)\\
= \,& c_3^q s^{-\frac{nq}{2}} \frac{2q}{c_4 s} \int_0^\infty \m(B_r(x_0)) r e^{-\frac{qr^2}{c_4 s}} dr\\
\leq \,& c_1 c_3^q s^{-\frac{nq}{2}} \frac{2q}{c_4 s} \int_0^\infty  r^{n+1} e^{-\left(\frac{q}{c_4 s}-c_2\right) r^2} dr\\
= \, &  c_1 c_3^q s^{\frac{n}{2}(1-q)} \frac{q}{c_4 } \Gamma\left(\frac{n}{2}+1\right) \left(\frac{q-c_2sc_4}{c_4 }\right)^{-\left(\frac{n}{2}+1\right)}\\
\leq \, &  \frac{c_1 c_3^q q }{ c_4} \Gamma\left(\frac{n}{2}+1\right) \left(\frac{q-c_2T_1c_4}{c_4}\right)^{-\left(\frac{n}{2}+1\right)} s^{\frac{n}{2}(1-q)}
\end{align}
for which we used the assumption $T_1<\frac{q}{c_4 c_2}$. Consequently inequality \eqref{eq:lemfirst} is proved with
\begin{equation}\label{eq:C1}
C_1(q,c_1,c_2,c_3,c_4,n,T_1) := \left(\frac{c_1 c_3^q q \Gamma(\frac{n}{2}+1)}{ c_4\left(\frac{q-c_2T_1c_4}{c_4}\right)^{\frac{n}{2}+1}}\right)^{\frac{1}{q}}.
\end{equation}
It follows that if $\frac{1}{q} + \frac{2}{mq'}>1$ then
\begin{align}
\|p.(x_0,\cdot)\|_{L^q_{q'}([0,t]\times X)} \leq \,& C_1 \left(\int_0^t s^{\frac{nq'}{2}(\frac{1}{q}-1)}ds\right)^{1/q'}\\
\leq \,& C_1 \left({\frac{nq'}{2}\left(\frac{1}{q}-1\right)+1}\right)^{-1/q'} t^{\frac{n}{2}(\frac{1}{q}-1)+\frac{1}{q'}}
\end{align}
and consequently inequality \eqref{eq:lemsecond} is proved with
\begin{equation}
C_2(q,q',c_1,c_2,c_3,c_4,n,T_1) := C_1 \left({\frac{nq'}{2}\left(\frac{1}{q}-1\right)+1}\right)^{-1/q'}
\end{equation}
as required.
\end{proof}

\begin{theorem}\label{thm:lpestv}
Suppose $V$ is non-negative with $\|V\|_{L^p(X)} < \infty$ for some $p>1$. Assume (\textbf{A1}) and (\textbf{A2}) and suppose $\gamma= 1-\frac{n}{2p} >0$. Denote by $T_0$ the constant determined by assumption (\textbf{A1}), choose some $T_1 < T_0 \wedge \frac{q}{c_4 c_2}$ where $q = \frac{p}{p-1}$ and define the constant $C_1$ by \eqref{eq:C1}. Then for each $\rho>0$ it follows that
\begin{equation}
\E^x \left[\exp\left(\int_0^t V(X_s)ds\right)\right] \leq \exp\left({\rho\left( {t \max\Bigg\lbrace \frac{1}{T_0}, \frac{c_4 c_2}{q}, \left(\frac{C_1\|V\|_{L^p(X)}}{\gamma (1-e^{-\rho})}\right)^{1/\gamma}\Bigg\rbrace +1}\right)}\right)
\end{equation}
for all $t\geq 0$ and all $x \in \supp \m$.
\end{theorem}

\begin{proof}
Choose $T_2>0$ sufficiently small so that
\begin{equation}
T_2 < T_0 \wedge \frac{q}{c_4 c_2} \wedge \left(\frac{\gamma (1-e^{-\rho})}{C_1\|V\|_{L^p(X)}}\right)^{1/\gamma}.
\end{equation}
By H\"{o}lder's inequality, Tonelli's theorem and Proposition \ref{prop:lqhk} we have
\begin{align}
\E^{x} \left[ \int_0^{t} V(X_s)ds\right] \leq\,&  \|V\|_{L^p(X)} \int_0^t \|p_s(x,\cdot)\|_{L^q(X)}ds \\
\leq\,&C_1 \|V\|_{L^p(X)} \frac{t^\gamma}{\gamma}\label{eq:hold0} \\ 
\leq\,& 1-e^{-\rho}\label{eq:hold}
\end{align}
for all $t \in [0,T_2]$. As an aside, note that if $\supp \m = X$ then, by inequality \eqref{eq:hold0}, such $V$ belong, under the present assumptions, to the \emph{Kato class} of $(X,d,\m)$, as in \cite{Simon1982}. The canonical diffusion has the property that if $A$ is a Borel subset with $\m(A)=0$ then $\mathbb{P}^x(X_t \in A)=0$ for all $t>0$. Therefore Khasminksii's lemma, as in \cite{FitzsimmonsPitman1999}, implies for a non-negative measurable function $V:X \rightarrow \R$ that if there exists a constant $0\leq \alpha <1$ such that
\begin{equation}
\sup_{x \in \supp \m} \E^{x} \left[\int_0^{t} V(X_s) ds\right] \leq \alpha
\end{equation}
then
\begin{equation}
\E^{x}\left[ \exp\left(\int_0^{t} V(X_s) ds\right)\right] \leq \frac{1}{1 - \alpha}
\end{equation}
for all $t\geq 0$ and $x \in \supp \m$. Consequently, by \eqref{eq:hold}, we have
\begin{equation}
\E^{x} \left[\exp\left(\int_0^{t} V(X_s) ds\right)\right] \leq e^{\rho}
\end{equation}
for all $t \in [0,T_2]$. By the Markov property, applied iteratively, we see that
\begin{align}
\E^x\left[\exp\left({\int_0^t V(X_s)ds}\right)\right] \leq\,& \left(\sup_{x \in \supp \m}\E^x\left[\exp\left({\int_0^{T_2} V(X_s)ds}\right)\right] \right)^{\floor*{\frac{t}{T_2}}+1}\\[2mm]
\leq\,& \exp\left({{\rho\left( {\frac{t}{T_2}+1}\right)}}\right)
\end{align}
for any $T_2$ of the form
\begin{equation}
T_2 = \beta\left( T_0 \wedge \frac{q}{c_4 c_2} \wedge \left(\frac{\gamma (1-e^{-\rho})}{C_1\|V\|_{L^p(X)}}\right)^{1/\gamma}\right)
\end{equation}
with $\beta \in (0,1)$. The result follows, by letting $\beta \uparrow 1$.
\end{proof}

Suppose now, more generally, that $V$ is a measurable function that can be decomposed as $V = V_1 - V_2$ with $V_1$ bounded below and $V_2$ non-negative with $\|V_2\|_p := \|V_2\|_{L^p(X)} < \infty$ for some $p>1$. Assume (\textbf{A1}) and (\textbf{A2}) and suppose $\gamma = 1-\frac{n}{2p} >0$. Denote by $T_0$ the constant determined by assumption (\textbf{A1}), set $q = \frac{p}{p-1}$ and choose some $T_1 < T_0 \wedge \frac{q}{c_4 c_2}$. Then, according to Theorem \ref{thm:lpestv}, for each $\rho>0$ there exists a positive constant $C_3(\rho):= C_3(q,c_1,c_2,c_3,c_4,n,T_0,T_1,\rho)$ such that
\begin{equation}\label{eq:intbound}
\E^x \left[\exp\left(\int_0^t V_2(X_s)ds\right)\right] \leq \exp\left( \rho\left(  t\, C_3(\rho) \left(1\vee\|V_2\|_{p}^{1/\gamma}\right) +1 \right)\right)
\end{equation}
for all $t\geq 0$ and all $x \in \supp \m$. Following \cite{Carmona1978}, the exit time estimate Theorem \ref{thm:etemmm} together with the bound \eqref{eq:intbound} can be used to derive an estimate on the expectation of the functional $e^{-\int_0^t V(X_s)ds}$. For each $a >0$ set $V_1^a(x) := \inf \{ V_1(y): d(y,x) \leq a\}$.

\begin{corollary}\label{cor:upperbndcor}
Assume (\textbf{A1}) and (\textbf{A2}) and suppose $1-\frac{n}{2p} >0$. Then for each $\rho>0$ there exists a positive constant $C_3(\rho)$ such that
\begin{align}
\E\left[\exp\left({-\int_0^t V(X_s)ds}\right)\right] \leq\,& \exp\left({\frac{\rho}{2}\left(  t\, C_3(\rho) \left(1\vee2\|V_2\|_{p}^{1/\gamma}\right) +1 \right)}\right)\\
& \times \left(e^{-2tV_1^a(x)} + e^{-2t \inf V_1} (1-\delta)^{-\frac{{N+\lambda}}{2}}\exp\left({-\frac{\delta a^2}{2t\Lambda(t)} }\right)\right)^{1/2}
\end{align}
for all $\delta \in (0,1)$, $t,a>0$ and $x \in \supp \m$, where $\lambda := \frac{1}{2}\sqrt{(N-1)K^-}.$
\end{corollary}

\begin{proof}
Using the bound \eqref{eq:intbound} obtained in the previous subsection, we see that
\begin{align}
&\E\left[\exp\left({-\int_0^t V(X_s)ds}\right)\right]^2\\
\leq\,& \exp\left({\rho\left(  t\, C_3(\rho) \left(1\vee 2\|V_2\|_{p}^{1/\gamma}\right) +1 \right)}\right) \E^x \left[ e^{-2\int_0^t V_1(X_s) ds}\right]\\[2mm]
\leq \,& \exp\left({\rho\left(  t\, C_3(\rho) \left(1\vee 2\|V_2\|_{p}^{1/\gamma}\right) +1 \right)}\right) \bigg( \E^x \left[ e^{-2\int_0^t V_1(X_s) ds}\, \mathbf{1}_{\lbrace \sup_{s\in \left[0,t\right]} d(X_s,x) < a\rbrace} \right]\\[2mm]
& \hspace{20mm}+ \E^x \left[ e^{-2\int_0^t V_1(X_s) ds}\, \mathbf{1}_{\lbrace \sup_{s\in \left[0,t\right]} d(X_s,x) \geq a\rbrace} \right]\bigg)\\[2mm]
\leq \,&  \exp\left({\rho\left(  t\, C_3(\rho) \left(1\vee 2\|V_2\|_{p}^{1/\gamma}\right) +1 \right)}\right) \bigg(e^{-2tV_1^a(x)} + e^{-2t \inf V_1} \mathbb{P}^x \bigg\lbrace \sup_{s\in \left[0,t\right]} d(X_s,x) \geq a\bigg\rbrace\bigg)
\end{align}
for each $t\geq 0$ and $a >0$. By Theorem \ref{thm:etemmm} we have
\begin{equation}
\mathbb{P}^x \bigg\lbrace \sup_{s\in \left[0,t\right]} d(X_s,x) \geq a\bigg\rbrace \leq (1-\delta)^{-\frac{{N+\lambda}}{2}}\exp\left(-\frac{\delta a^2}{2t{\Lambda}(t)}\right)
\end{equation}
for all $t>0$ and $\delta \in (0,1)$ which by substituting yields the desired inequality.
\end{proof}

Observe if $f \in L^2(X;\m)$ then H\"{o}lder's inequality, Proposition \ref{prop:lqhk} and Theorem \ref{thm:lpestv} imply, under the assumptions of Theorem \ref{thm:lpestv}, that for each $t>0$ the map
\begin{equation}
f \mapsto \E^x\left[f(X_t) e^{-\int_0^t V(X_s)ds}\right]
\end{equation}
is bounded on $L^2(X;\m) \rightarrow L^\infty(X;\m)$.

\subsection{Upper bounds for Schr\"{o}dinger eigenfunctions}\label{ss:upperbnds}

As pointed out by a referee, under the present assumptions $V$ belongs to the Kato class of $(X,d,\mathfrak{m})$ (see the proof of Theorem \ref{thm:lpestv}), and so a self-adjoint extension of $-\tfrac{1}{2}\Delta + V$ can be defined by quadratic form methods, as in \cite{Guneysu2019}. Given an eigenfunction $\varphi \in \mathcal{D}\left(-\tfrac{1}{2}\Delta + V\right)$, satisfying
\begin{equation}
-\left(\tfrac{1}{2}\Delta - V\right)\varphi = E \varphi
\end{equation}
for some $E \in \mathbb{R}$, it follows that $\varphi$ is given by the Feynman-Kac formula
\begin{equation}\label{eq:fk}
\varphi(x) = e^{tE} \E^x \left[ \varphi(X_t) \exp\left(-\int_0^t V(X_s) ds\right]\right)
\end{equation}
for all $t \geq 0$. This formula can now be used to estimate $\varphi$. In particular, following \cite{Carmona1978}, we observe that certain growth conditions on $V$ imply pointwise decay estimates for the eigenfunction. For a function $f$ and $a>0$ we set
\begin{equation}
f^a(x) := \inf \lbrace f(y) : y \in B_a(x)\rbrace.
\end{equation}
Then, for example, we observe that if $\lim_{d(x,x_0) \rightarrow \infty} V_1(x) = \infty$ and if $V_1$ satisfies, for some positive function $a$ on $X$, the condition
\begin{equation}
V_1^a(x) \geq \alpha V_1(x)
\end{equation}
for some positive constant $\alpha$ and for $x$ outside a compact set, then there must exist positive constants $c$ and $C$, depending on $E$ and $\|V_2\|_p$, such that
\begin{equation}
|\varphi(x)| \leq C \|\varphi\|_\infty e^{-c a(x)\sqrt{V_1(x)}}
\end{equation}
for all $x$ outside the compact set with $x\in \supp \m$. Indeed, this follows from the representation \eqref{eq:fk} and Corollary \ref{cor:upperbndcor} by setting $t(x) = a(x) V_1^{-\frac{1}{2}}(x)$ and using the fact that $\Lambda(t) \geq 1$ for all $t \geq 0$. We also have the following proposition:

\begin{proposition}
Set $\lambda := \frac{1}{2}\sqrt{(N-1)K^-}$. Assume (\textbf{A1}) and (\textbf{A2}), suppose $1-\frac{n}{2p} >0$. If for some $x_0 \in X$ we have
\begin{equation}
V(x) \geq \gamma d^{2m}(x,x_0)
\end{equation}
outside a compact set for some constants $\gamma>0$ and $m\geq 1$, then for all
\begin{equation}
\theta < \frac{\sqrt{\gamma} m^m}{2(m+1)^{m+1}}
\end{equation}
there exists a constant $C>0$ such that
\begin{equation}\label{eq:decay}
|\varphi(x)| \leq  C \|\varphi\|_\infty e^{- \theta d^{m+1}(x,x_0)}
\end{equation}
for all  $x \in \supp \m$.
\end{proposition}

\begin{proof}
Since $\varphi$ is bounded, it suffices to prove \eqref{eq:decay} for $d(x,x_0)$ large enough. For $\beta >0$ and $0 < \alpha < 1$ set $a = \alpha d(x,x_0)$ and $t = \beta \gamma^{-\tfrac{1}{2}} d^{-(m-1)}(x,x_0)$ to obtain from Corollary \ref{cor:upperbndcor} that
\begin{align}
|\varphi(x)|\leq\,& \|\varphi\|_\infty  \exp\left({ tE+ \frac{\rho}{2}\left(  t\, C_3(\rho) \left(1\vee2\|V_2\|_{p}^{1/\gamma}\right) +1 \right)}\right)\\
& \hspace{20mm} \times \left(e^{-2tV^a(x)} + (1-\delta)^{-\frac{{N+\lambda}}{2}}\exp\left(-\frac{\delta \sqrt{\gamma} \alpha^2}{2 \beta}d^{m+1}(x,x_0) \right)\right) 
\end{align}
for all $\delta \in (0,1)$. Then note, for this choice of $a$, by the reverse triangle inequality
\begin{align}
V^a(x) = \gamma \inf_{y \in B_a(x)} d^{2m}(y,x_0)\geq \gamma \inf_{y \in B_a(x)} |d(y,x)-d(x,x_0)|^{2m}= \gamma (1-\alpha)^{2m} d^{2m}(x,x_0)
\end{align}
and therefore, taking for example $\delta = \tfrac{1}{2}$, we find
\begin{align}
|\varphi(x)|\leq\,& \|\varphi\|_\infty 2^{\frac{{N+\lambda}}{2}} \exp\left({tE+ \frac{\rho}{2}\left(  t\, C_3(\rho) \left(1\vee2\|V_2\|_{p}^{1/\gamma}\right) +1 \right)}\right)\\
&\hspace{10mm}\times \left(\exp\left(-\beta \sqrt{\gamma} (1-\alpha)^{2m} d^{m+1}(x,x_0)\right)+ \exp\left(-\frac{\sqrt{\gamma} \alpha^2}{4 \beta}d^{m+1}(x,x_0) \right)\right)\\
\leq\,& \|\varphi\|_\infty 2^{\frac{{N+\lambda}}{2}} \exp\left(\beta \gamma^{-\frac{1}{2}} d^{-(m-1)}(x,x_0)\left(E + \frac{\rho}{2} C_3(\rho) \left(1\vee2\|V_2\|_{p}^{1/\gamma}\right)\right)\right)\\
&\hspace{20mm}\times \exp\left(\frac{\rho}{2}-\theta d^{m+1}(x,x_0)\right)
\end{align}
for any
\begin{equation}
\theta < \min\bigg\lbrace \beta \sqrt{\gamma} (1-\alpha)^{2m},\frac{\sqrt{\gamma} \alpha^2}{4 \beta}\bigg\rbrace.
\end{equation}
Since
\begin{equation}
\max \bigg\lbrace \min\bigg\lbrace \beta \sqrt{\gamma} (1-\alpha)^{2m},\frac{\sqrt{\gamma} \alpha^2}{4 \beta}\bigg\rbrace : \alpha \in (0,1), \beta >0 \bigg\rbrace = \frac{\sqrt{\gamma} m^m}{{2(m+1)}^{(m+1)}}
\end{equation}
and since $m\geq 1$, it follows that there exists a constant such that the claim is satisfied.
\end{proof}

Returning to the setting of Section \ref{sec:subriem}, note that similar eigenfunction estimates, for operators of the form $\tfrac{1}{2}\Delta_{\mathcal{H}} - V$, can also be obtained (under certain conditions), by Theorem \ref{thm:exittime}.

\section*{Acknowledgements}
This work has been supported by Fonds National de la Recherche
Luxembourg (FNR), project O14/7628746 GEOMREV.

\end{document}